\documentclass[a4paper]{amsart}

\usepackage{amssymb,latexsym}
\usepackage{amsmath}
\usepackage{amsfonts}
\usepackage{amsthm}
\usepackage{amssymb, comment}

\theoremstyle{plain}
\newtheorem{theorem}{Theorem}
\newtheorem{lemma}[theorem]{Lemma}
\newtheorem{proposition}[theorem]{Proposition}

\def\Z{\mathbb{Z}}
\def\Q{\mathbb{Q}}

\theoremstyle{definition}
\newtheorem{definition}{Definition}

\newtheorem{remark}{Remark}

\title[$D(n)$-quintuples with square elements]
{$D(n)$-quintuples with square elements}
\begin{document}

\date{}

%\footnotesize\date{\today}

\author[A. Dujella]{Andrej Dujella}
\address{
Department of Mathematics\\
Faculty of Science\\
University of Zagreb\\
Bijeni{\v c}ka cesta 30, 10000 Zagreb, Croatia
}
\email[A. Dujella]{duje@math.hr}

\author[M. Kazalicki]{Matija Kazalicki}
\address{
Department of Mathematics\\
Faculty of Science\\
University of Zagreb\\
Bijeni{\v c}ka cesta 30, 10000 Zagreb, Croatia
}
\email[M. Kazalicki]{matija.kazalicki@math.hr}

\author[V. Petri\v{c}evi\'c]{Vinko Petri\v{c}evi\'c}
\address{
Department of Mathematics\\
Faculty of Science\\
University of Zagreb\\
Bijeni{\v c}ka cesta 30, 10000 Zagreb, Croatia
}
\email[V. Petri\v{c}evi\'c]{vpetrice@math.hr}

\begin{abstract}
For an integer $n$, a set of $m$ distinct nonzero integers $\{a_1,a_2,\ldots,a_m\}$
such that  $a_ia_j+n$ is a perfect square for all $1\leq i<j\leq m$, is called a 
$D(n)$-$m$-tuple. In this paper, we show that there are infinitely many essentially different $D(n)$-quintuples with square elements.
We obtained this result by constructing genus one curves on a certain double cover of $\mathbb{A}^2$ branched along four curves.

\end{abstract}


\maketitle

\section{Introduction}

For an integer $n$, a set of $m$ distinct nonzero integers
with the property that the product of any two of its distinct
elements plus $n$ is a square, is called a Diophantine $m$-tuple with the property $D(n)$ or $D(n)$-$m$-tuple. The $D(1)$-$m$-tuples (with rational elements) are called simply (rational) Diophantine $m$-tuples, and have been studied since the ancient time. 

The first example of a rational Diophantine quadruple was the set
$$
\left\{\frac{1}{16},\, \frac{33}{16},\, \frac{17}{4},\, \frac{105}{16}\right\}
$$
found by Diophantus. Fermat found the first Diophantine quadruple in integers $\{1,3,8,120\}$. Euler proved that the exist
infinitely many rational Diophantine quintuples (see \cite{Hea}),
in particular he was able to extend the integer Diophantine quadruple found by Fermat,
to the rational quintuple
$$
\left\{ 1, 3, 8, 120, \frac{777480}{8288641} \right\}.
$$
Stoll \cite{Stoll} recently showed that this extension is unique.

In 1969, using linear forms in logarithms of algebraic numbers and a reduction method
based on continued fractions, Baker and Davenport \cite{B-D}
proved that if $d$ is a positive integer such that
$\{1, 3, 8, d\}$ forms a Diophantine quadruple, then $d$ has to be $120$.
This result motivated the conjecture that there does not exist a Diophantine quintuple in integers.
The conjecture has been proved recently by He, Togb\'e and Ziegler \cite{HTZ}
(see also \cite{BTF,duje-crelle}).

On the other hand, it is not known how large can a rational Diophantine tuple be.
In 1999, Gibbs found the first example of rational Diophantine sextuple \cite{Gibbs1}
$$
\left\{ \frac{11}{192}, \frac{35}{192}, \frac{155}{27}, \frac{512}{27}, \frac{1235}{48}, \frac{180873}{16} \right\}.
$$
In 2017, Dujella, Kazalicki, Miki\'c and Szikszai \cite{DKMS} proved that there are
infinitely many rational Diophantine sextuples, while
Dujella and Kazalicki \cite{Duje-Matija} (inspired by the work of Piezas \cite{P})
described another construction of parametric families of rational Diophantine sextuples.
Recently, Dujella, Kazalicki and Petri\v{c}evi\'c in \cite{DKP-sext}
proved that there are infinitely many rational Diophantine sextuples
such that denominators of all the elements (in the lowest terms) in the sextuples are perfect squares, and in \cite{DKP-reg} they proved that there are infinitely many Diophantine sextuples containing two regular quadruples and one regular quintuple.
No example of a rational Diophantine septuple is known.
The Lang conjecture on varieties of general type implies that
the number of elements of a rational Diophantine tuple is bounded by an absolute constant
(see the introduction of \cite{DKMS}). Diophantine $m$-tuples have been studied over the rings other that $\Z$ and $\Q$, for example Dujella and Kazalicki \cite{DK-finite} computed
the number of Diophantine quadruples over finite fields.
For more information on Diophantine $m$-tuples see the survey article \cite{Duje-notices}.

Sets with $D(n)$ properties have also been extensively studied. It is easy to show that there are no integer $D(n)$-quadruples if $n\equiv 2 \pmod{4}$, and it is know that if $n \not \equiv 2 \pmod{4}$ and $n \not \in \{-4,-3,-1,3,5,8,12,20\}$, then there is at least one $D(n)$-quadruple \cite{Duje-generalization}. Recently, Bonciocat, Cipu and Mignotte \cite{BCM} proved that there are no $D(-1)$-quadruples (as well as $D(-4)$-quadruples) thus leaving the existence of $D(n)$-quadruples in the remaining six sporadic cases open. 

Dra\v zi\'c and Kazalicki \cite{DrazicKazalicki} described rational $D(n)$-quadruples with fixed product of elements in terms of points on certain elliptic curves. It is not known if there is a rational Diophantine $D(n)$-quintuple for every $n$, and no example of rational $D(n)$-sextuple is known if $n$ is not a perfect square. 

One can also study $m$-tuples that have $D(n)$-property for more than one $n$.  Ad\v zaga, Dujella, Kreso and Tadi\'c \cite{ADKT} presented several families of Diophantine triples which have $D(n)$-property for two distinct $n$'s with $n\ne 1$ as well as some Diophantine triples which are $D(n)$-sets for three distinct $n$'s with $n\ne 1$.  Dujella and Petri\v cevi\'c in \cite{DP} proved that there are infinitely many (essentially different) integer quadruples which are simultaneously $D(n_1)$-quadruples and $D(n_2)$-quadruples with $n_1 \ne n_2$, and in \cite{DP2} showed that the same thing is true for three distinct $n$'s (since the elements of their quadruples are squares one of $n$'s is equal to zero). Our main result extends the previous results to quintuples.

\begin{theorem}\label{thm:main}
 There are infinitely many nonequivalent quintuples that have $D(n_1)$ property for some $n_1\in \mathbb{N}$ such that all the elements in the quintuple are perfect squares. 
 In particular, there are infinitely many nonequivalent integer quintuples that are simultaneously $D(n_1)$-quintuples and $D(n_2)$-quintuples with $n_1 \ne n_2$ since then we can take $n_2=0$.
\end{theorem}

Note that if $\{a,b,c,d,e\}$ is a $D(n_1)$-quintuple, and $u$ a nonzero rational, then 
$\{u a, u b, u c, u d, u e\}$ is a $D(n_1 u^2)$-quintuple and we say that these two quintuples are equivalent. Since every rational Diophantine quintuple is equivalent to some $D(u^2)$-quintuple whenever $u$ is an integer divisible by the common denominator of the elements in the quintuple, Theorem \ref{thm:main} will follow if we prove that there are infinitely many rational Diophantine quintuples with the property that the product of any two of its elements is a perfect square.

A Diophantine quadruple $\{a,b,c,d\}$ is called regular if 
$$(a+b-c-d)^2=4(ab+1)(cd+1).$$ 

\begin{definition}
	We say that rational Diophantine quintuple $\{a,b,c,d,e\}$ is {\it exotic} if $abcd=1$, quadruples $\{a,b,d,e\}$ and $\{a,c,d,e\}$ are regular, and if the product of any two of its elements is a perfect square.  
\end{definition}

Denote by $S$ an affine surface defined over $\Q$ by
{\small
\begin{equation*}
 (1 + r - 2 r^2 t - t^2 + r t^2)(-1 + r + 2 r^2 t + t^2 + r t^2)(-r - r^2 - 2t - r t^2 + r^2 t^2)(r - r^2 - 2 t + r t^2 + r^2 t^2)  = y^2.
\end{equation*}
}

Define a rational map $p:S \rightarrow \mathbb{A}^5$ given by $p(r,t,y)=(a,b,c,d,e)$ where
\begin{align*}
	a&=\frac{(r^2-1)(t^2-1)(s^2-1)}{8rst},\\
	b&=\frac{2(t^2-1)rs}{(r^2-1)(s^2-1)t},\\
	c&=\frac{2(s^2-1)rt}{(r^2-1)(t^2-1)s},\\
	d&=\frac{2(r^2-1)st}{(s^2-1)(t^2-1)r},\\
\end{align*}
where $s=\frac{-1 + r^2 + t + r^2 t}{-1 - r^2 - t + r^2 t}$ and $e$ is defined by formula \eqref{eq:e1} from Section \ref{sec:story}.

\begin{proposition}\label{prop: param}
For every exotic quintuple $(a,b,c,d,e)$ there is a rational point $(r,t,y)\in S(\Q)$ on curve $S$ such that $(a,b,c,d,e)=p(r,t,y)$.
Conversely, if $(r,t,y)\in S(\Q)$ is a rational point on $S$ in a domain of $p$, then $p(r,t,y)$ is exotic quintuple provided that all elements are distinct and nonzero.
\end{proposition}
Note that one can explicitly determine the degeneracy locus of map $p$ - a finite set of curves on $S$ such that for every $(r,t,y)\in S(\Q)$ which is not on any of those curves we have that $p(r,t,y)$ is exotic quintuple. Thus, any curve on $S$ with an infinite number of rational points will give rise to infinitely many exotic quintuples.

Denote by $\pi:S \rightarrow \mathbb{A}^2$ the projection $\pi(r,t,y)=(r,t)$. Let $D_1, D_2$ and $D_3$ be plane genus zero curves in $\mathbb{A}^2$ defined by
\begin{align*}
	D_1&: r^2t^2 - 4r^2t - 3r^2 - 2rt^2 - 2r - 3t^2 - 4t + 1 = 0,\\
	D_2&: r^2 t - r^2 + 2 r t^2 + 2 r - t - 1 =0,\\
	D_3&:  r^2 t^2 + 3 r^2 - t^2 + 2t - 1 = 0,
\end{align*}
and by $\widetilde{D_i}=\pi^{-1}(D_i)$ pullbacks of these curves to $S$ via $\pi$.

\begin{proposition}\label{prop:main}
Curves $\widetilde{D_1}, \widetilde{D_2}$ and $\widetilde{D_3}$ are genus one curves defined over $\Q$ birationally equivalent to the elliptic curves with positive Mordell-Weil rank. In particular, there are infinitely many exotic quintuples.	
\end{proposition}

For an example, consider the following parametrization of $\widetilde{D_3}$
$$(r,t)=\left(-\frac{2 u+1}{u^2+u+1},\frac{u^2+4 u+1}{(u-1) (u+1)}\right).$$
It defines a curve birational to $\widetilde{D_3}$ given by quartic
$$-48 \left(u^2-3 u-1\right) \left(u^2+5 u+3\right)=v^2.$$
The point $(u,v)=(3,36)$ of this quartic corresponds to $(r,t)=(-\frac{7}{13},\frac{11}{4})$ which in turn is mapped by $p$ to the Diophantine quintuple 
$$M=\left\{\frac{225^2}{480480},  \frac{2548^2}{480480},\frac{286^2}{480480},\frac{1408^2}{480480},\frac{819^2}{480480}\right\}.$$
With this quintuple our investigation begun.

\section{Parametrizing exotic quintuples} \label{sec:story}

Our starting point was experimental discovery of a exotic rational Diophantine quintuple $M$ (defined in the introduction)
which by clearing denominators gives Diophantine $D(480480^2)$-quintuple with square elements.
This quintuple $\{a,b,c,d,e\}$ has the following structure which motivated our construction of infinite families
\begin{itemize}
	\item[i)] $abcd=1$
	\item[ii)] quadruples $\{a,b,d,e\}$ and $\{a,c,d,e\}$ are regular.
\end{itemize}

\begin{proposition}\label{prop:abcd}
	Let $\{a,b,c,d\}$ be a rational Diophantine quadruple with $abcd=1$. Then there exist $r,s,t \in \Q$ such that 
	$$a=xyz,\quad  b=\frac{x}{yz},\quad c=\frac{y}{xz}, \quad d=\frac{z}{xy},$$
	where $x=\frac{t^2-1}{2t}, y=\frac{s^2-1}{2s}$ and $z=\frac{r^2-1}{2r}$. In particular,
	the product of any two elements of the quadruple is a perfect square.
\end{proposition} 
\begin{proof}
From $ab+1=ab+abcd=ab(1+cd)$ it follows that $ab$ is a perfect square, and similarly for other products.
Set $ab=x^2, ac=y^2$ and $ad=z^2$, with $x,y,z \in \Q$. It follows  $a^2=\frac{ab\cdot ac}{bc}=\frac{x^2 y^2}{1/z^2}$, hence
$a=xyz$ and similarly $b=\frac{x}{yz}$, $c=\frac{y}{xz}$ and $d=\frac{z}{xy}$ (with the appropriate choice of signs).
Since $x^2+1$ is a perfect square, there is $t\in \Q$ such that $x=\frac{t^2-1}{2t}$, and similarly for $y$ and $z$. The claim follows. 
\end{proof}

To extend quadruple $\{a,b,c,d\}$ defined by $r,s,t\in \Q$ (as in Proposition \ref{prop:abcd}) to an exotic quintuple it is enough that
triples $\{a,b,d\}$ and $\{a,c,d\}$ have a common regular extension $e$, and that $ae$ is a perfect square. 

Since both $\{a,b,d\}$ and $\{a,c,d\}$ extend to regular quadruples in two different ways, $\{e_1,e_1'\}$ and $\{e_2,e_2'\}$ respectively, to check if there is a common regular extension a priori we have four conditions to inspect.
It is easy to see that the maps $\sigma_1(r,s,t)=(1/r,s,1/t)$ and $\sigma_2(r,s,t)=(\frac{1}{r},s,-t)$ are symmetries of the equations from Proposition \ref{prop:abcd} defining $a,b,c$ and $d$, hence both $(r,s,t)$ and $\sigma_i(r,s,t)$ give rise to the same quadruple $(a,b,c,d)$. In general the map $(r,s,t) \mapsto (a,b,c,d)$ is $32:1$, but we will not need the whole group of symmetries. Moreover, $\sigma_1$ ``maps'' $e_2$ to $e_2'$ and fixes $e_1$, while $\sigma_2$ maps $e_1$ to $e_1'$ and fixes $e_2$. Therefore, to parametrize quadruples with common regular extension as above it is enough to solve $e_1=e_2$ for any choice of $e_1$ and $e_2$.
Thus for the choice of $e_1$ and $e_2$ 
{
\begin{equation}\label{eq:e1}e_1=\frac{u_1(r,s,t)u_2(r,s,t)u_3(r,s,t)u_4(r,s,t)}{8 (-1 + r) r (1 + r) (-1 + s) s (1 + s) (-1 + t) t (1 + t)},
\end{equation}
where 
\begin{align*}u_1(r,s,t)&=-1 - r + s - r s - t - r t - s t + r s t,\\ u_2(r,s,t)&=1 + r - s + r s - t - 
r t - s t + r s t,\\ u_3(r,s,t)&=1 - r - s - r s + t - r t + s t + r s t,\\ u_4(r,s,t)&=-1 + 
r + s + r s + t - r t + s t + r s t\\
\end{align*}
}
and 
{
	$$e_2=\frac{v_1(r,s,t)v_2(r,s,t)v_3(r,s,t)v_4(r,s,t)}{8 (-1 + r) r (1 + r) (-1 + s) s (1 + s) (-1 + t) t (1 + t)},$$
}
where 
\begin{align*}
v_1(r,s,t)&=-1 - r - s - r s + t - r t - s t + r s t,\\ v_2(r,s,t)&=1 + r - s - r s - t + 
r t - s t + r s t,\\ v_3(r,s,t)&=1 - r + s - r s - t - r t + s t + r s t,\\ v_4(r,s,t)&=-1 + 
r + s - r s + t + r t + s t + r s t\\
\end{align*}
we obtain the following condition
$$(s - t) (1 + s t) (1 - r^2 - s - r^2 s - t - r^2 t - s t + r^2 s t) (-1 - r^2 + s - r^2 s + t - r^2 t + s t + r^2 s t)=0.$$

Solutions to $(s - t)(1 + st)=0$ induce degenerate quintuples (with zero element or with two identical elements) so we can ignore them.
To reduce the argument further, note that $a,b,c,d,e_1$ and $e_2$ are fixed by the map $\sigma_3(r,s,t)=(-1/r,s,-1/t)$. Moreover, $\sigma_3$
defines a birational map between affine plane surfaces defined by $1 - r^2 - s - r^2 s - t - r^2 t - s t + r^2 s t=0$ and $-1 - r^2 + s - r^2 s + t - r^2 t + s t + r^2 s t=0$
which is an isomorphism outside the vanishing set of $rst(r^2-1)(s^2-1)(t^2-1)=0$. Since the triples $(r,s,t)$ from this vanishing set do not correspond to Diophantine quintuples, 
without the loss of generality we can assume that the triples $(r,s,t)$ describing rational Diophantine quintuples $\{a,b,c,d,e\}$, such that $abcd=1$ and that both $\{a,b,d,e\}$ and $\{a,c,d,e\}$ are regular, satisfy
\begin{equation}\label{eq:regularity}
1 - r^2 - s - r^2 s - t - r^2 t - s t + r^2 s t=0.
\end{equation}

On the other hand, the condition that $ae_1$ is a perfect square is equivalent to
\begin{align*}
	&(-1 - r + s - rs - t - rt - st + rst)(1 + r - s + rs - t - rt - st + rst)\\&(1 - r - s - rs + t - rt + st + rst)(-1 + r + s + rs + t - rt + st + rst) = y^2.
\end{align*}

Substituting $s$ from \eqref{eq:regularity} we obtain a defining equation for the affine surface $S$ defined in the introduction. Thus, we have constructed a rational map from the introduction
$$p:S \rightarrow \mathbb{A}^5,\quad p(r,t,y)=(a,b,c,d,e),$$ 
(defined by \eqref{eq:e1},  \eqref{eq:regularity} and  Proposition \ref{prop:abcd}) and proved that for every exotic quintuple $(a,b,c,d,e)$ there is a rational point $(r,t,y)$ on the surface $S$ such that $p(r,t,y)=(a,b,c,d,e)$. Note that the pair $(r,t)$ defining the point is not necessary unique.

Conversely, given $(r,t,y)\in S(\Q)$ such that $p(r,t,y)$ is defined, the quintuple $p(r,t,y)$ will be exotic if it is non-degenerate (all elements must be distinct and nonzero). This finishes the proof of Proposition \ref{prop: param}.

\section{Construction of curves on $S$}

If we show that the surface $S$ has infinitely many rational points outside the degeneracy locus of $p$ (a finite set of curves on $S$ whose rational points either map under $p$ to degenerate quintuples or $p$ is not defined for them), then Proposition \ref{prop: param} will imply Theorem \ref{thm:main}. For that, we will construct genus one curves on $S$ which are defined over $\Q$ and birationaly equivalent to elliptic curves (over $\Q$) of positive Mordell-Weil rank.

Denote by $\pi:S \rightarrow \mathbb{A}^2$ the projection $\pi(r,t,y)=(r,t)$,
and denote by
\begin{align*}
	 C_1&: 1 + r - 2 r^2 t - t^2 + r t^2=0,\\ C_2&:-1 + r + 2r^2 t + t^2 + r t^2=0,\\ C_3&: -r - r^2 - 2t - r t^2 + r^2 t^2=0,\\ C_4&:r - r^2 - 2 t + r t^2 + r^2 t^2=0,
\end{align*}	 
curves over which the map $\pi$ is ramified.

The configuration of these curves has a large symmetry group. One can readily check that the maps
\begin{align*}
	\tau_1(r,t)&=(-r,t), \quad \tau_2(r,t)=\left(\frac{1}{r},\frac{1}{t}\right),\\ \tau_3(r,t)&=\left(-r,\frac{t-1}{t+1}\right),
\end{align*}
extend to the birational automorphisms of $C_1\cup C_2 \cup C_3 \cup C_4$, and also to the birational automorphisms of $S$. While we have already encountered maps $\tau_1$ and $\tau_2$, note that if $p(r,t)=(a,b,c,d,e)$, then
$p(\tau_3(r,t))=(-d,-c,-b,-a,-e)$. They generate a group $G$ of order $16$ which acts on the set of plane curves $D \subset \mathbb{A}^2$. We say that two plane curves $D_1$ and $D_2$ are equivalent if there is $\tau \in G$ such that it generates birational map from $D_1$ to $D_2$.

\begin{remark}
	We remark one curiosity related to $\tau_3$. Note that if $(a,b,c,d)$ is rational Diophantine quadruple with $abcd=1$ which corresponds to the triple $(r,s,t)$, then rational Diophantine quadruple $(\frac{1}{a},\frac{1}{b},\frac{1}{c}, \frac{1}{d})$ corresponds to the triple $(\frac{r-1}{r+1}, \frac{s-1}{s+1}, \frac{t+1}{t-1})$.
\end{remark}

Let $D\subset \mathbb{A}^2$ be a plane curve of genus zero, and denote by $\widetilde{D}=\pi^{-1}(D)$ a pullback of $D$ under $\pi$. Assume that $\widetilde{D}$ is absolutely irreducible of genus $g$. Genus $g$ is controlled by the ramification of $\pi|\widetilde{D}$. More precisely, if we resolve singularities of the projective closures of $\widetilde{D}$ and $D$, and apply Riemman-Hurwitz formula to the corresponding extension $\widetilde{\pi}$ of $\pi|\widetilde{D}$ we will get
$$2g-2=-4+N,$$
where $N$ is the number of ramification points of $\widetilde{\pi}$. In particular, if we want $g$ to be one, then $N$ must be equal to four.

Denote by $L=\bigcup_{i\ne j} C_i \cap C_j$. We have the following characterization of ramification points of $\widetilde{\pi}$.
\begin{lemma}
	Assume that, for some $i$, $C_i$ and $D$ intersect transversally at $P$. If $P \notin L$ and if $P$ is nonsingular on $D$, then $\widetilde{\pi}$ is ramified at $P$.
\end{lemma}

The previous lemma suggests that if we want to search for a genus zero plane curve $D$ for which $\widetilde{D}=\pi^{-1}(D)$ is genus one curve, our best candidates would be curves that
intersect  $\cup_i C_i$ outside $L$ in as few points possible. This task gets harder as the degree of $D$ gets bigger - by B\'ezout's theorem $D$ and $C_i$ intersect at $3\deg{D}$ or $4\deg{D}$ points (counting multiplicities and points at infinity). Also, to control the genus of $D$ one needs to specify singularities (whose number is described by Pl\"ucker's formula) which a priori can be anywhere (but it works best for us if they are on $\cup C_i$ since then this intersection will probably not count for ramification) so this made systematic computer search impossible for us to implement. 

In addition to this approach, in order to employ the symmetry group $G$, we also searched for curves $D$ on which some $\tau \in G$ induces birational automorphism. 
The logic behind this is that if, for example, such $D$ intersects $C_i$ (ideally) in $L$, and if $\tau$ induced birational map between $C_i$ and $C_j$, then $D$ also intersects $C_j$ in $L$ (since $\tau$ maps $L$ into itself).

\section{Results}

In our computer search we found three unequivalent genus zero curves $D_1,D_2, D_3 \subset \mathbb{A}^2$ defined in the introduction
such that curves $\widetilde{D_i}=\pi^{-1}(D_i)$ are genus one curves birational to the elliptic curve with positive Mordell-Weil rank.
Interestingly, the orbit of each of these curves under the action of $G$ is of size $8$ - the curves are fixed by elements $\tau_2, \tau_2 \circ \tau_3^2$ and $\tau_1$ respectively.
The following analysis of curves $\widetilde{D_i}$ finishes the proof of Proposition \ref{prop:main}.

\subsection{Curve $\widetilde{D_1}$}
We have the following parametrization of the curve $D_1$
$$\psi_1(u)= \left(-\frac{3 u^2+4 u-1}{3 u^2+8 u+3},\frac{2 u}{(u+1) (3 u+1)}\right),$$
which gives the following model for $\widetilde{D_1}$
$$-192 \left(5 u^2+10 u+3\right) \left(9 u^2+18 u-1\right)=v^2,$$
which is birational to the elliptic curve
$$E_1: y^2=x^3 - 2892x - 59024,$$
of rank $1$ and torsion subgroup isomorphic to $\Z/2\Z \times \Z/2\Z$.
A generator of infinite order in $E_1(\Q)$ corresponds to $u=-60/233$ and parameterizes the following exotic quintuple
\begin{align*}
	a&=-\frac{29529940110878678717653}{420081952495961042800800},\quad b=-\frac{3041992513146972959}{11548847964
	0779256},\\c&=-\frac{351416293757343837}{2249352029178441082},\quad d=-\frac{1776863948138083954777600}{5140041910
	12768208630559},\\e&=-\frac{927643283361539913482847}{141804790226710724159200}.
\end{align*}
\subsection{Curve $\widetilde{D_2}$}
We have the following parametrization of the curve $D_2$
$$\psi_2(u)= \left(\frac{u (2 u+1)}{(u+1) (u+2)},-\frac{u^2-2 u-2}{u (u+2)}\right),$$
which gives the following model for $\widetilde{D_2}$
$$48 \left(u^2+16 u+10\right) \left(3 u^2-2\right)=v^2,$$
which is birational to the elliptic curve
$$E_2:  y^2 = x^3 - 876x - 9520,$$
of rank $1$ and torsion subgroup isomorphic to $\Z/2\Z \times \Z/2\Z$.
A generator of infinite order in $E_2(\Q)$ corresponds to $u=\frac{113}{23}$ and parameterizes the following exotic quintuple
\begin{align*}
	\Bigl\{&-\frac{482493852225}{293535838544},-\frac{1058592509345792}{1212259417081713},-\frac{120747048705
	6449}{74793264945984},\\&-\frac{18858398366873}{437001310622800},-\frac{695331110026639}{116388239242275
}\Bigr\}.
\end{align*}
\subsection{Curve $\widetilde{D_3}$}
We have the following parametrization of the curve $D_3$
$$\psi_3(u)=\left(-\frac{2 u+1}{u^2+u+1},\frac{u^2+4 u+1}{(u-1) (u+1)}\right),$$
which gives the following model for $\widetilde{D_3}$
$$-48 \left(u^2-3 u-1\right) \left(u^2+5 u+3\right)=v^2,$$
which is birational to the elliptic curve
$$E_3:   y^2 + x y + y = x^3 - x^2 - 41 x + 96,$$
of rank $1$ and torsion subgroup isomorphic to $\Z/2\Z \times \Z/2\Z$.
A generator of infinite order in $E_3(\Q)$ corresponds to $u=-4$ and parameterizes the following exotic quintuple
$$
\left\{-\frac{5632}{1365},-\frac{4459}{330},-\frac{143}{840},-\frac{3375}{32032},-\frac{2457}{1760}\right\}.
$$

\section{Concluding remarks}
While we have found infinitely many rational Diophantine quintuples with $D(0)$ property, it remains open if there is a rational Diophantine quintuple with square elements.
On the other hand, there are infinitely many rational Diophantine quadruples with square elements, for example the following two parametric family has this property
\begin{align*}
a&=\frac{3^2(s-1)^2(s+1)^2v^2}{2^2(2s^3-2s+v^2)^2},\\
b&=\frac{v^2(-4s^3+4s+v^2)^2}{2^2(s+1)^2(s-1)^2(-s^3+s+v^2)^2},\\
c&=\frac{(2s^3-2s+v^2)^2}{3^2 v^2 s^2},\\
d&=\frac{4^2(-s^3+s+v^2)^2 s^2}{v^2 (-4 s^3+4 s+v^2)^2}.
\end{align*}
This family is obtained by taking $t=1/(r-1)$ in the notation of Proposition \ref{prop:abcd}.
We have also found an example of a rational Diophantine quadruple with square elements for which the product $abcd \ne 1$
$$\left\{\left(\frac{18}{77}\right)^2, \left(\frac{55}{96}\right)^2, \left(\frac{56}{15}\right)^2, \left(\frac{340}{77}\right)^2\right\}.$$

\begin{comment}
due to Lasi\'c \cite{luka}, which is symmetric in the three involved parameters:
\begin{align*}
a_1 &=  \frac{2 t_1 (1 + t_1 t_2 (1 + t_2 t_3))} {(-1 + t_1 t_2 t_3) (1 + t_1 t_2 t_3)}, \\
a_2 &=  \frac{2 t_2 (1 + t_2 t_3 (1 + t_3 t_1))} {(-1 + t_1 t_2 t_3) (1 + t_1 t_2 t_3)}, \\
a_3 &=  \frac{2 t_3 (1 + t_3 t_1 (1 + t_1 t_2))} {(-1 + t_1 t_2 t_3) (1 + t_1 t_2 t_3)}.
\end{align*}
\end{comment}

{\bf Acknowledgements.}
The authors would like to thank Goran Dra\v zi\'c for the simple proof of Proposition \ref{prop:abcd}.
The authors were supported by the Croatian Science Foundation under the project no.~IP-2018-01-1313.
The authors acknowledge support from the QuantiXLie Center of Excellence, a project
co-financed by the Croatian Government and European Union through the
European Regional Development Fund - the Competitiveness and Cohesion
Operational Programme (Grant KK.01.1.1.01.0004).

\end{document}